\documentclass[preprint,12pt]{article}

\usepackage{amsmath,amssymb,amsthm,epsfig}

\newtheorem{theorem}{Theorem}[section]
\begin{document}

\title{A problem involving the $p$-Laplacian operator}
\author{Ratan K. Giri \& D. Choudhuri\footnote{Corresponding
author: dc.iit12@gmail.com} \\
}
\date{}
\maketitle

\begin{abstract}
\noindent Using a variational technique we guarantee the existence
of a solution to the \emph{resonant Lane-Emden} problem $-\Delta_p
u=\lambda |u|^{q-2}u$, $u|_{\partial\Omega}=0$ if and only if a
solution to $-\Delta_p u=\lambda |u|^{q-2}u+f$,
$u|_{\partial\Omega}=0$, $f \in L^{p'}(\Omega)$ ($p'$ being the
conjugate of $p$) , exists for $q\in (1,p)\bigcup (p,p^{*})$ under a
certain condition for both the cases, i.e., $1<q<p<p^{*}$ and $1< p < q < p^{*}$ - the sub-linear and the super-linear cases. \\
{\bf keywords}:~ $p$-laplacian; elliptic PDE; Palais-Smale condition; Sobolev space.\\
{\bf AMS classification}:~35A15, 35A01.
\end{abstract}

\section{Introduction}
The study of partial differential equations involving a
$p$-laplacian differential operator has become a major case of study
in the recent times although it is still far from being completely
understood, especially when $p=1$ or $\infty$. A few evidences of
the limiting case can be found in \cite{bhatta}, \cite{lind}. When
$p=2$, the usual Laplacian is obtained for which a vast literature
exists (\cite{Evans}, \cite{kesavan} and the references therein).
For $p \neq 2$ the p-Laplace operator has physical applications in
the study of non-Newtonian fluids (dilatant fluids when
$p>2$)~\cite{nikos}. In practical life most of the problems are non
linear by nature for which a numerical solution is seeked for,
however, unearthing the existence of solution leads to a rich theory
hidden behind the partial differential equation. The problems we are
going to address in this article are the following. Let $\Omega$ be
a bounded subset of $\mathbb{R}^n$, $n \geq 3$ with a Lipschitz
boundary $\partial\Omega$. Given $1 < p < \infty$ and $q\in
(1,p)\bigcup (p,p^{*})$, where $p^{*}=\frac{np}{n-p}$ if $1<p<n$ and
$p^{*}=\infty$ if $p \geq n$, we consider the following problems.
\begin{enumerate}
\item
$-\Delta_p u = \lambda|u|^{q-2}u$, $u|_{\partial\Omega}= 0$. This problem is also known as the resonant Lane-Emden problem.
\item $-\Delta_p u = \lambda|u|^{q-2}u+f$, $f  \in L^{p'}(\Omega)$, $u|_{\partial\Omega}= 0$.
\end{enumerate}
where $\lambda$ is a real number, $\Delta_{p}=\nabla\cdot(|\cdot|^{p-2}\nabla\cdot)$. Throughout this paper we shall refer the problems in 1 and 2 as the first and the second problem respectively.
\\
\noindent We call the first problem to be of sub-critical type if $1<q<p<p^{*}$ and of super-critical type when $p^{*}>q>p>1$. It is found in~\cite{Grey} that a unique solution exists to the first problem for the sub-critical case whereas uniqueness is lost for the super-critical case. Readers interested in knowing more about the first problem can refer to examples found in~\cite{Drabek},~\cite{Garcia}, where the domain is ring shaped for $q \sim p^{*}$ and the solution is non-unique. Kawohl~\cite{kawo} showed the same but the domain which was considered is of annulus type with the annulus being sufficiently small in size. Dancer~\cite{Dancer} showed that if $p=2$ and $\Omega$ is a general domain then a unique solution exists to the first problem. Uniqueness is also guaranteed in~\cite{Diaz} for the sub-linear case whereas a subdifferential method has been used to prove existence in~\cite{Otani} for both sub and super critical cases.\\
In this paper we will use a well known variational technique to show the existence of a solution in $W_0^{1,p}(\Omega)=\{v \in L^p(\Omega):\nabla v\in L^{p}(\Omega), v|_{\partial\Omega}=0\}$. A Fredholm type alternative is also proposed thus showing a connection between the first and the second problem. We organize the paper into two sections. In Section 2 we give the Mathematical formulation. In Section 3 we discuss a few preliminary results and the main result.

\section{Mathematical formulation}
The following definitions and theorems will be used in the main result we prove.\\
{\bf 2.1 Definition}:~ Let $X$ be a Banach space and $H:X \rightarrow \mathbb{R}$ a $C^{1}$ functional. It is said to satisfy the {\it Palais-Smale condition} (PS) if the following holds:\\
 Whenever $\{u_n\}$ is a sequence in $X$ such that $\{H(u_n)\}$ is bounded and $H^{'}(u_n) \rightarrow 0$ strongly in $X^{'}$ (the dual space), then $\{u_n\}$ has a strongly convergent subsequence.\\
The (PS) condition is a strong condition as very ``well-behaved" function do not satisfy it ({\it Example}: $f(x)=c$, $x \in \mathbb{R}$, $c$ a real constant).\\
We now state the following important theorem due to Ambrosetti and Rabinowitz~\cite{Ambro} which is a common tool used in the theory of modern PDEs.\\\\
{\it Mountain-pass theorem}:~Let $H:X \rightarrow \mathbb{R}$ be a $C^{1}$ functional satisfying (PS). Let $u_0$, $u_1 \in X$, $c_0 \in \mathbb{R}$ and $r > 0$ such that
\begin{enumerate}
\item $||u_1-u_0|| > r$
\item $H(u_0)$, $H(u_1) < c_0 \leq H(v)$, $\forall v$ such that $||v-u_0||=r$. Then $H$ has a critical value $c \geq c_0$ defined by
\begin{eqnarray}
    c &=& \inf\limits_{\Gamma \in \wp}\max\limits_{t \in [0,1]}H(\Gamma(t))
\end{eqnarray}
where $\wp$ is the collection of all continuous paths $\Gamma: [0,1] \rightarrow X$ such that $\Gamma(0)=u_0$, $\Gamma(1)=u_1$.
\end{enumerate}
{\bf 2.2 Weak formulation of the problem}:~We now give the weak formulation of the first problem. We say that $u \in W_0^{1,p}(\Omega)$ is a weak solution of the first problem if
\begin{eqnarray}
\int_{\Omega}|\nabla u|^{p-2}\nabla u\cdot\nabla vdx -\lambda\int_{\Omega}|u|^{q-2}uv dx &=& 0
\end{eqnarray}
for every $v \in W_0^{1,p}(\Omega)$.\\
The weak solutions of the Lane-Emden problem are the critical points of the energy function defined by
\begin{eqnarray}
J_{q}(u)&=&\frac{1}{p}\int_{\Omega}|\nabla u|^{p}dx-\frac{\lambda}{q}\int_{\Omega}|u|^{q}dx.
\end{eqnarray}
The following compact embedding theorems, due to Rellich-Kondrasov have been used in our work.
\begin{eqnarray}
\text{if}~p < n, W_0^{1,p}(\Omega) \hookrightarrow L^{q}(\Omega),~1 \leq q < p^{*},\nonumber\\
\text{if}~p=n, W_0^{1,n}(\Omega) \hookrightarrow L^{q}(\Omega),~1 \leq q < \infty,\nonumber\\
\text{if}~p>n, W_0^{1,p}(\Omega) \hookrightarrow C(\overline{\Omega}).~~~~~~~~~~~~~~~~~~\nonumber
\end{eqnarray}
We consider the non-homogeneous counterpart of the first problem - which is the second problem - and is as follows.
\begin{eqnarray}
-\Delta_p u &=& \lambda|u|^{q-2}u+f, \nonumber\\
u|_{\partial\Omega}&=& 0,
\end{eqnarray}
where $f \in L^{p^{'}}(\Omega)$, $p^{'}$ being the conjugate of $p$
and is equal to $\frac{p}{p-1}$. Let the corresponding functional be
denoted by $J$ which is defined as follows.
\begin{eqnarray}
J(u) &=& \frac{1}{p}\int_{\Omega}|\nabla u|^{p}dx-\frac{\lambda}{q}\int_{\Omega}|u|^qdx -\int_{\Omega}fudx.
\end{eqnarray}
The Fr\'{e}chet derivative of $J$, which is in
$W_0^{-1,p^{'}}(\Omega)$ where $p^{'}=\frac{p}{p-1}$, is
\begin{eqnarray}
<J^{'}(u),v> &=& \int_{\Omega}|\nabla u|^{p-2}\nabla u\cdot\nabla vdx -\lambda\int_{\Omega}|u|^{q-2}uvdx -\int_{\Omega}fvdx,
\end{eqnarray}
$\forall v \in W_0^{1,p}(\Omega)$. Thus $u \in W_0^{1,p}(\Omega)$ is a weak solution of the second problem if
\begin{eqnarray}
\int_{\Omega}|\nabla u|^{p-2}\nabla u\cdot\nabla vdx -\lambda\int_{\Omega}|u|^{q-2}uvdx -\int_{\Omega}fvdx &=&0. \nonumber
\end{eqnarray}
For the sake of further analysis we redefine the functional as follows.
\begin{eqnarray}
J_q(u)&=& -\chi_{(1,p)}(q)J(u)+\chi_{(p,p^{*})}(q)J(u),
\end{eqnarray}
where $\chi$ is the indicator function. From the sections which follow we shall use the functional in (7).
\section{Few preliminary results and the main theorem} The main result of this paper is as follows. The problem
$-\Delta_{p}u = \lambda|u|^{q-2}u$, $u|_{\partial\Omega}= 0$
has a weak solution if and only if the problem $-\Delta_{p}u=\lambda|u|^{q-2}u+f$, $u|_{\partial\Omega}=0$, where $f \in L^{p/p-1}(\Omega)$, has a weak solution. We prove the result for $p<n$. The case of $p \geq n$ follows the same proof as in the case $p<n$ which is based on the results on compact embedding stated after equation (3). But first we present a few technical lemmas on which the proof of this result relies upon.\\
We first assume that a solution exists to the problem
\begin{eqnarray}
-\Delta_{p}u &=& \lambda|u|^{q-2}u,\nonumber\\
u|_{\partial\Omega}&=&0.
\end{eqnarray}

\begin{theorem}
The mapping $J_{q}$ defined in (7) is a $C^1$-functional over
$W_0^{1,p}(\Omega)$.
\end{theorem}
\begin{proof}
We first prove that the functional $J'$ is continuous which will
imply that $J_q'$ is continuous and hence the theorem will follow.
Consider
\begin{eqnarray}
|<J^{'}(u),v>| & \leq & \int_{\Omega}|\nabla u|^{p-2}\nabla u\cdot\nabla vdx+|\lambda|\int_{\Omega}|u|^{q-1}vdx+\int_{\Omega}|f||v|dx\nonumber\\
& \leq &  ||\nabla u||_{\frac{p}{p-1}}||\nabla v||_{p}+|\lambda|||u||_{\frac{q}{q-1}}||v||_{q}+||f||_{\frac{p}{p-1}}||v||_{p}\nonumber\\
&\leq & \left[||\nabla u||_{\frac{p}{p-1}}+C_1|\lambda|||u||_{\frac{q}{q-1}}+C_2||f||_{\frac{p}{p-1}}\right]||\nabla v||_{p},~\forall v \in W_0^{1,p}(\Omega),
\end{eqnarray}
where $C_1$, $C_2$ are the constants due to the embedding of $W_0^{1,p}(\Omega)$ in $L^{q}(\Omega)$ for $q\in [1,p^{*}]$.
From (8)\&(9) one can see that $J$ is a $C^1$ functional over $W_0^{1,p}(\Omega)$.
\end{proof}
\begin{theorem}
There exists $u_0, u_1 \in W_0^{1,p}(\Omega)$ and a positive real
number $c_0$ such that $J_q(u_0),J_q(u_1)<c_0$ and $J_q(v)\geq c_0$,
for every $v$ satisfying $||v-u_0||_{1,p}=r$.
 \end{theorem}
 \begin{proof}
 Let $u_0=0$. Clearly $u_0$ is a solution of (8) and $J_q(0)=0$. Now let $w \in B(0,1)$ in $W_0^{1,p}(\Omega)$ and consider $v=u_0+rw$ for $r > 0$ and
 hence $||v-u_0||_{1,p} = r$. We first show the existence of $r$ such that $||v-u_0||_{1,p}=r_0$ and for which $J(v)\geq c_0$ for each $v\in B(0,r_0)$.\\
 Let $p<q<p^{*}$. Now
\begin{eqnarray}
J_q(u_0+r w)-J_q(u_0) &=& \frac{r^{p}}{p}\int_{\Omega}|\nabla w|^pdx-\frac{r^{q}\lambda}{q}\int_{\Omega}|w|^qdx-r\int_{\Omega}fwdx,\nonumber\\
&=& \frac{r^{p}}{p}-\frac{r^{q}\lambda}{q}\int_{\Omega}|w|^qdx-r\int_{\Omega}fwdx.
\end{eqnarray}
Further, $|w|_{1,p}=1$ and hence $|\int_{\Omega}w^{p}dx| \leq \int_{\Omega}|w^{p}|dx \leq c||w||^{p}_{p} \leq c_1|w|_{1,p}=c_1$. Similarly, $|\int_{\Omega}w^{q}dx| \leq c_2$. Using these arguments leads to
\begin{eqnarray}
J_q(u_0+r w)-J_q(u_0) & \geq &  r\left[\frac{r^{p-1}}{p}-\frac{r^{q-1} \lambda}{q}c_2-c_1^{1/p}||f||_{p'}\right],\nonumber\\
& = & c'.
\end{eqnarray}
We first analyze the term $\left[\frac{r^{p-1}}{p}-\frac{r^{q-1}
\lambda}{q}c_2-c_1^{1/p}||f||_{p'}\right]=F(r)$ (say). Clearly
$F(0)<0$ and for $r_0=\left(\frac{q(p-1)}{p(q-1)}\frac{1}{\lambda
c_2}\right)^{\frac{1}{q-p}}$ we see that $F'(r_0)=0$. A bit of
calculus guarantees that $F''(r_0)<0$ and hence $r_0$ is a maximizer
of $F$. If
$0<\lambda<\lambda_1=\frac{q(p-1)}{p(q-1)}.\left(\frac{p(q-1)}{q-p}.c_1^{\frac{1}{p}}||f||_{p'}
\right)^{\frac{1}{1-p}}$ then $F(r_0)>0$. As $r\rightarrow\infty$ we
have $F(r)\rightarrow -\infty$. Hence there exists $r_1, r_2>0$ and
$r_1<r_0<r_2$ such that $F(r)>0$ for each $r\in (r_1,r_2)$. We
choose $r=r_0$ such that $||v-u_0||_{1,p}=r_0$ and for which
$J_q(v)\geq c'$ for each $v\in B(0,r_0)$. \noindent Similarly, if
$1<q<p$ then according to the definition of $J_q$ we now have
\begin{eqnarray}
J_q(u_0+r w)-J_q(u_0) & = & -J(u_0+rw)+J(u_0)\nonumber\\
 & \geq &r\left[-\frac{r^{p-1}}{p}+\frac{r^{q-1} \lambda}{q}c_2+c_1^{1/p}||f||_{p'}\right],\nonumber\\
& = & c^{''}.
\end{eqnarray}
Using the same argument as for the case of $p<q<p^{*}$ we find $r$
and $0<\lambda<\lambda_2$ such that $J_q(v)\geq c^{''}$ for all
$||v-u_0||=r$. We choose $\lambda'=\min\{\lambda_1, \lambda_2\}$
such that $0<\lambda<\lambda'$ and $c_0=\min\{c^{'}, c^{''}\}$.
\\
{\it Choice of $u_1$}:~Let $w_p$ be the first eigen vector of $-\Delta_{p}$, i.e., $-\Delta_{p}w_p=\lambda_p |w_p|^{p-1}w_p$, where $\lambda_p$ is the first eigen value of $-\Delta_{p}$. The first eigen value of the $p$-laplacian operator is strictly positive \cite{ly}. Consider the function $g=k w_p$, $k\in\mathbb{R}$, $||w_p||_{1,p}=1$ and $p<q<p^{*}$. Note that,
\begin{eqnarray}
J_q(g) & = & \left(\frac{k^p}{p}-\frac{\lambda k^q \int_{\Omega}|w_p|^{q}dx}{q}\right)-kC,\nonumber
\end{eqnarray}
where $C=\int_{\Omega}fw_pdx$. Since $p<q<p^{*}$, we observe $k$ can
be chosen arbitrarily large so that
$\displaystyle{\frac{k_0^p}{p}-\frac{\lambda
k_0^q\int_{\Omega}|w_p|^{q}dx}{q}- k_0C< 0}$.
 Then $J_q(kw_p) < 0$ and hence $J_q(kw_p)<J_q(u_0)$. Thus we can choose $u_1=k_0w_p$, where $k_0>r_0$. Then $||u_1-u_0||_{1, p}>r_0$. Similarly for $1<q<p$ we have
 \begin{eqnarray}
J_q(g) & = & \left(-\frac{k^p}{p}+\frac{\lambda k^q \int_{\Omega}|w_p|^{q}dx}{q}\right)+kC,\nonumber
\end{eqnarray}
and $k$ can be chosen large enough to make $J_q(g)<0$. Hence the result.
 \end{proof}
\begin{theorem}
$J_q$ satisfies the Palais-Smale condition.
\end{theorem}
\begin{proof}
Let us consider the case for which $p<q<p^*$. The other case for
$1<q<p$ follows similarly. Let $u_n$ be a sequence in
$W_0^{1,p}(\Omega)$ such that $|J_q(u_n)|\leq M$ and
$J_q^{'}(u_n)\rightarrow 0$ as $n\rightarrow\infty$ in
$W_0^{-1,p^{'}}(\Omega)$, $p'$ being the conjugate of $p$. Now
\begin{eqnarray}
J_q(u_n) &=& \frac{1}{p}\int_{\Omega}|\nabla u_n|^{p}dx-\frac{\lambda}{q}\int_{\Omega}|u_n|^{q}dx -\int_{\Omega}fu_ndx,\\
<J_q^{'}(u_n),v> &=& \int_{\Omega}|\nabla u_n|^{p-2}\nabla u_n.\nabla vdx -\lambda\int_{\Omega}|u_n|^{q-2}u_n vdx -\int_{\Omega}fvdx, \forall v \in W_0^{1,p}(\Omega).\nonumber\\
\end{eqnarray}
Consider the following.
\begin{eqnarray}
<J_q^{'}(u_m),u_m> &=& \int_{\Omega}|\nabla u_m|^pdx-\lambda\int_{\Omega}|u_m|^{q}dx-\int_{\Omega}fu_mdx,\\
J_q(u_m)&=& \frac{1}{p}\int_{\Omega}|\nabla u_m|^{p}dx-\frac{\lambda}{q}\int_{\Omega}|u_m|^{q}dx-\int_{\Omega} fu_mdx,\nonumber\\
&=& \frac{1}{p}|u_m|_{1,p}^{p}-\frac{\lambda}{q}\int_{\Omega}|u_m|^qdx-\int_{\Omega}fu_mdx\nonumber\\
\lambda\int_{\Omega}|u_m|^{q}dx &=& \frac{q}{p}|u_m|_{1,p}^{p}-qJ_q(u_m)-q\int_{\Omega}fu_mdx,\nonumber\\
\frac{p-q}{p} |u_m|_{1,p}^{p} &=&
<J_q^{'}(u_m),u_m>-qJ_q(u_m)-q\int_{\Omega}fu_mdx.
\end{eqnarray}
This implies that $|u_m|_{1,p}$ is bounded.
The above inequality in (14) clearly shows that $u_n$ is bounded in $W_0^{1,p}(\Omega)$ and hence by Eberlein-\v{S}mulian's theorem (refer Dunford-Schwartz [1; p. 430]~\cite{dunford}) it has a {\it weakly} convergent subsequence, say $u_{n_{k}}$, in $W_0^{1,p}(\Omega)$.\\
\noindent{\bf Claim}.The sequence $\{u_{n_{k}}\}$ is strongly convergent in $W_0^{1,p}(\Omega)$.\\
\noindent{\bf Proof}.~Applying limit $k\rightarrow\infty$ to (14)
(refer Appendix) and using the strong convergence of $(u_{n_{k}})$
in $L^q(\Omega)$ due to compact embedding we obtain
\begin{eqnarray}
\int_{\Omega} |\nabla u|^{p-2}\nabla u\cdot\nabla vdx &=& \lambda\int_{\Omega} |u|^{q-2}uvdx+\int_{\Omega}fvdx,
\end{eqnarray}
and we pass on the limit to (15) we get
\begin{eqnarray}
\lim\limits_{n\rightarrow\infty}|u_{n_{k}}|_{1,p}^{p} = \lambda\int_{\Omega}|u|^{q}dx+\int_{\Omega}fudx = |u|_{1,p}^{p}.
\end{eqnarray}
Thus a weakly convergent sequence which is convergent in norm is strongly convergent. Hence $u_{n_{k}}\rightarrow u$ in $W_0^{1,p}(\Omega)$ as $k\rightarrow\infty$.
\end{proof}
\noindent Thus  by the Mountain-pass theorem an extreme point for $H$ exists in $W_0^{1,p}(\Omega)$ \newline
We summarize the results proved in Theorems $3.1$, $3.2$ and $3.3$ in the form of a theorem as follows.
\begin{theorem}
Suppose $-\Delta_p u = \lambda|u|^{q-2}u$, $u|_{\partial\Omega}= 0$ has a solution. Then
\begin{enumerate}
\item the functional $J_q=-\chi_{(1,p)}(q)J(u)+\chi_{(p,p^{*})}(q)J(u)$ where $J(u)=\frac{1}{p}\int_{\Omega}|\nabla u|^{p}dx-\frac{\lambda}{q}\int_{\Omega}|u|^qdx -\int_{\Omega}fudx$ is $C^1$ and satisfies the Palais-Smale condition,
\item $J_q$ satisfies the hypothesis of the Mountain-Pass theorem.
\end{enumerate}
Therefore $J_q$ has an extreme point in $W_0^{1,p}(\Omega)$. In
other words $-\Delta_p u = \lambda|u|^{q-2}u+f$, $\,f \in
L^{p'}(\Omega)$, $u|_{\partial\Omega}= 0$ has a solution whenever
$\lambda\in(0,\lambda']$ where $\lambda'=\min\{\lambda_1,
\lambda_2\}$ as found in Theorem 3.2.


\end{theorem}
\noindent Conversely, suppose a solution to the problem
\begin{eqnarray}
-\Delta_p u&=&\lambda|u|^{q-2}u+f,~f \in L^{p'}(\Omega),\nonumber\\
u|_{\partial\Omega}&=&0.
\end{eqnarray}
We subdivide this situation into two different cases - namely, $1<q<p$ (the sub-linear case) and $1<p<q<p^{*}$ (the super-linear cases).
Let $(f_n)\subset L^{p^{'}}(\Omega)$ be a sequence such that $f_n\rightarrow 0$ in $L^{p'}(\Omega)$. By the assumption, to each $f_n$ there exists a solution,
say $u_n$.\\
We have {\bf $q\in (1,p)\bigcup(p,p^{*})$} and
\begin{eqnarray}
B[u,v]&=&\int_{\Omega}|\nabla u|^{p-2}\nabla u\cdot\nabla vdx-\lambda\int_{\Omega}|u|^{q-2}uvdx,\nonumber\\
&=&\int_{\Omega} fvdx,\,\,\forall v\in W_0^{1,p}(\Omega)
\end{eqnarray}
where $B$ is a `{\it non linear form}' in two variables $u$ and $v$. It is easy to check that $B(.,.)$ is the Fr\'{e}chet derivative of the $C^1$ functional $\frac{1}{p}\int_{\Omega}|\nabla u|^p-\frac{\lambda}{q}\int_{\Omega}|u|^q$ and hence is continuous.\\
Clearly, for each $ v\in W_0^{1,p}(\Omega)$ we have
\begin{eqnarray}
B[u_n,v]&=& \int_{\Omega}|\nabla u_n|^{p-2}\nabla u_n\cdot\nabla vdx-\lambda\int_{\Omega}|u_n|^{q-2}u_nvdx  ,\nonumber\\
&=&\int_{\Omega}f_nvdx,\nonumber\\
& \leq & ||f_n||_{p^{'}}||v||_{p} \rightarrow 0\,\,\mbox{as}\,\, n\rightarrow \infty.
\end{eqnarray}
Hence $\int_{\Omega}f_nvdx\rightarrow 0$ as $n\rightarrow\infty$. 
Consider $T_n(v)= \displaystyle{\int_{\Omega}|\nabla u_n|^{p-2}\nabla u_n\cdot\nabla vdx}$. Then $T_n$ is bounded linear over $W_0^{1,p}(\Omega)$ and $||T_n||= |||\nabla u_n|^{p-1}||_{p'}$.
From the above definition of $T_n$, for a fixed $v\in W_0^{1,p}(\Omega)$ we have the sequence $(T_n(v))$ to be bounded which implies that $(T_n(v))$ is pointwise bounded. Thus by the uniform boundedness principle $(||T_n||)$ is bounded.
Thus $||\nabla u_n||_{p}$ is bounded. Hence, there exists
a subsequence $(u_{n_k})$ which weakly converges to $u_{\infty}$
with respect to the $||\cdot||_{1,p}$ in $W_0^{1,p}(\Omega)$. Hence
we have
\begin{eqnarray}
\lim_{k \rightarrow \infty}\int_{\Omega}|\nabla v|^{p-2}\nabla v\cdot\nabla u_{n_{k}}dx &=& \int_{\Omega}|\nabla v|^{p-2}\nabla v \cdot\nabla u_{\infty}dx, \forall v \in W_0^{1,p}(\Omega).\nonumber\\
\Rightarrow\lim_{k \rightarrow \infty}\int_{\Omega}|\nabla
u_{n_{l}}|^{p-2}\nabla u_{n_{l}}\cdot\nabla u_{n_{k}}dx &=&
\int_{\Omega}|\nabla u_{n_l}|^{p-2}\nabla u_{n_{l}}\cdot\nabla
u_{\infty}dx,
\end{eqnarray}
for a fixed $l$.
Therefore, since $u_{n_{k}}\rightharpoonup u_{\infty}$ in $W_0^{1,p}(\Omega)$ implies that $|\nabla u_{n_{k}}|^{p-1}\rightharpoonup |\nabla u_{\infty}|^{p-1}$ (for a subsequence) in $L^{p'}(\Omega)$ (Refer Appendix). But $W_0^{1,p}(\Omega) \hookrightarrow L^{p}(\Omega)  \hookrightarrow W^{-1,p^{'}}(\Omega)$ and hence
\begin{eqnarray}
\lim_{l \rightarrow \infty}\int_{\Omega}|\nabla u_{n_{l}}|^{p-2}\nabla u_{n_{l}}\cdot\nabla vdx &=& \int_{\Omega}|\nabla u_{\infty}|^{p-2}\nabla u_{\infty}\cdot\nabla vdx, \forall v \in W_0^{1,p}(\Omega),\nonumber\\
\Rightarrow\lim_{l \rightarrow \infty}\int_{\Omega}|\nabla
u_{n_{l}}|^{p-2}\nabla u_{n_{l}}\cdot\nabla u_{\infty}dx &=&
\int_{\Omega}|\nabla u_{\infty}|^{p}dx .
\end{eqnarray}
Hence, $\lim_{k \rightarrow \infty}\int_{\Omega}|\nabla u_{n_{k}}|^{p}dx=\int_{\Omega}|\nabla u_{\infty}|^{p}dx$.
It immediately can be concluded that there exists a $u_{\infty}$ such that $u_n \rightarrow u_{\infty}$ in $W_0^{1,p}(\Omega)$. Hence using the continuity of $B[.,.]$ in (20) we have
\begin{eqnarray}
\lim_{n\rightarrow\infty}B[u_n,v]&=&\lim_{n\rightarrow\infty}\int_{\Omega}|\nabla u_n|^{p-2}\nabla u_n\cdot\nabla vdx-\lim_{n\rightarrow\infty}\lambda\int_{\Omega}|u_n|^{q-2}u_nvdx,\nonumber\\
&=&\lim_{n\rightarrow\infty}\int_{\Omega} f_nvdx,\nonumber\\
\Rightarrow B[u_{\infty},v]&=& 0, \forall v \in W_0^{1,p}(\Omega).\nonumber
\end{eqnarray}
In other words
\begin{eqnarray}
\int_{\Omega}|\nabla u_{\infty}|^{p-2}\nabla u_{\infty}\cdot\nabla vdx-\lambda\int_{\Omega}|u_{\infty}|^{q-2}u_{\infty}vdx&=& 0, \forall v \in W_0^{1,p}(\Omega).
\end{eqnarray}
\noindent We summarize the result proved as follows.
\begin{theorem}
Suppose $-\Delta_p u = \lambda|u|^{q-2}u+f$, $f\in L^{p'}(\Omega)$,
$u|_{\partial\Omega}= 0$ has a solution. If $u_n$ is a solution of the PDE corresponds to $f_n$, where $f_n \subset
L^{p'}(\Omega)$ such that $f_n \rightarrow 0$ in $L^{p'}(\Omega)$,  then  we have $B[u_{\infty},v]=0$ for each $v\in
W_0^{1,p}(\Omega)$ and thus $u_{\infty}$ is a solution to $-\Delta_p
u = \lambda|u|^{q-2}u$, $u|_{\partial\Omega}= 0$.
\end{theorem}

\section{Appendix}
We show that
\begin{eqnarray}
\lim\limits_{n\rightarrow\infty}\int_{\Omega}|\nabla
u_n|^{p-2}\nabla u_n\cdot\nabla vdx &=& \int_{\Omega}|\nabla
u|^{p-2}\nabla u\cdot\nabla vdx,~~ \forall v\in W_0^{1,p}(\Omega).
\end{eqnarray}
We divide the explanation into two cases:\\
\underline{Case 1}:~When $p>2$.\\
This implies that $p'$, the conjugate of $p$, should be lesser than $2$, i.e., $1<p'<2<p$. Thus we have $W_0^{1,p}(\Omega)\hookrightarrow_{compact} L^{p'}(\Omega)$ (since $W_0^{1,p}(\Omega)\hookrightarrow_{compact} L^{q}(\Omega)$ for $q\in [1,p^{*})$).\\
Since $\nabla u_n$ converges weakly to, say $\nabla u$, in $L^{p}(\Omega)$, hence $<|\nabla u_n|-|\nabla u|,v>\rightarrow 0$ for each $v\in L^{p'}(\Omega)$.
Thus $<|\nabla u_n|-|\nabla u|,|\nabla u_n|-|\nabla u|>\rightarrow 0$, i.e., $||\nabla u_n||_2\rightarrow ||\nabla u||_2$.
Hence $||\nabla u_n||_{p'}\rightarrow ||\nabla u||_{p'}$ because $p'<2<p$. By the Riesz-Fischer theorem~\cite{bachman}, there exists a subsequence of $\nabla u_n$
which converges pointwise a.e., i.e., $|\nabla u_n(x)|\rightarrow |\nabla u(x)|$. So $|\nabla u_n(x)|^{p-1}\rightarrow |\nabla u(x)|^{p-1}$ and
hence $|\nabla u_n|^{p-1}\rightharpoonup |\nabla u|^{p-1}$ in $L^{p'}(\Omega)$. Thus we have $\lim_{n \rightarrow \infty}\int_{\Omega}|\nabla u_n|^{p-2}\nabla u_n\cdot\nabla vdx=\int_{\Omega}|\nabla u|^{p-2}\nabla u\cdot\nabla vdx$, $\forall v \in W_{0}^{1,p}(\Omega)$.\\
\underline{Case 2}:~When $p<2$.\\
This implies that $p'$, the conjugate of $p$, should be greater than $2$, i.e., $p<2<p'$.\\
\noindent Look at the map $F:W_0^{1,p}(\Omega)\rightarrow L^{p'}(\Omega)$ defined by $u \mapsto |\nabla u|^{p-1}$. Consider the range of $F$, i.e., $R(F)=\{|\nabla u|^{p-1}:u\in W_0^{1,p}(\Omega)\}$.\\
Observe that the map $F$ is bounded in the sense that bounded sets are mapped to bounded sets. Hence if $u_n \rightharpoonup u$ in $W_0^{1,p}(\Omega)$ implies that $(u_n)$ is bounded in $W_0^{1,p}(\Omega)$. Hence $(F(u_n))=(|\nabla u_n|^{p-1})$ is bounded in $L^{p'}(\Omega)$. Since $L^{p'}(\Omega)$ is reflexive, hence there exists a subsequence of $|\nabla u_n|^{p-1}$ which weakly converges to, say, $w$ in $L^{p'}(\Omega)$.


\noindent We have the following.\\
$u_n\rightharpoonup u$ in $W_0^{1,p}(\Omega)$ so $|\nabla u_n|^{p-1}\rightharpoonup w$ in $L^{p'}(\Omega)$. This implies that
\begin{eqnarray}
<|\nabla u_n|^{p-1}-w,v>& \rightarrow & 0, \forall v\in L^{p}(\Omega)\nonumber
\end{eqnarray}
Since $p<2<p'$ hence $|\nabla u_n|^{p-1}-w\in L^{p}(\Omega)$. Thus $|||\nabla u_n|^{p-1}-w||_{2}\rightarrow 0$ and hence $|||\nabla u_n|^{p-1}-w||_p\rightarrow 0$. Therefore we have a subsequence of $(|\nabla u_n|^{p-1})$ such that $|\nabla u_{n}|^{p-1}\rightarrow w$ pointwise a.e. (implying $|\nabla u_{n}|\rightarrow w^{\frac{1}{p-1}}$ pointwise a.e.) and so $|\nabla u_{n}| \rightharpoonup w^{\frac{1}{p-1}}$ in $L^p(\Omega)$. Hence $w=|\nabla u|^{p-1}$.\\
Thus in all the above cases we found the following.
\begin{eqnarray}
\lim\limits_{n\rightarrow\infty}\int_{\Omega}|\nabla
u_n|^{p-2}\nabla u_n\cdot\nabla vdx &=& \int_{\Omega}|\nabla
u|^{p-2}\nabla u\cdot\nabla vdx,~~ \forall v\in W_0^{1,p}(\Omega).
\end{eqnarray}
Hence by the compact embedding due to Rellich-Kondrachov it can be concluded $u_n\rightarrow u$ in $L^q(\Omega)$. Thus we also have
\begin{eqnarray}
\lim\limits_{n\rightarrow\infty}\int_{\Omega}|u_n|^{q-2}u_n\cdot\nabla
vdx &=& \int_{\Omega}|u|^{q-2} u\cdot\nabla vdx,~~ \forall v\in
W_0^{1,p}(\Omega).
\end{eqnarray}

\section{Conclusions}
The resonant Lane-Emden problem has been studied. An existence result has been established to the non-homogeneous Lane-Emden problem for the sub-linear - $1<q<p<p^{*}$ and the super-linear case - $1<p<q<p^{*}$ for $\lambda\in (0,\lambda']$ - $\lambda'$ being sufficiently large - if it is assumed that a non-trivial solution exists to the homogeneous Lane-Emden problem for the sub-linear - $1<q<p<p^{*}$ and the super-linear case - $1<p<q<p^{*}$, which is basically an eigen value problem. We further proved the `{\it converse}' that if the non-homogeneous problem has a solution then a solution to the homogeneous problem exists for both the sub and the super critical cases.

\section{Acknowledgement}
 One of the authors (RKG) thanks the financial assistantship received from the Ministry of Human Resource Development (M.H.R.D.).

{\sc Ratan Kumar Giri} and {\sc D. Choudhuri}\\
Department of Mathematics,\\
National Institute of Technology Rourkela, Rourkela - 769008,
India

\end{document}